\documentclass[11pt,reqno]{amsart}

\usepackage{amsfonts, amsthm, amsmath}
\allowdisplaybreaks[4]

\usepackage{rotating}

\usepackage{tikz}

\usepackage{graphics}

\usepackage{amssymb}

\usepackage{amscd}

\usepackage[latin2]{inputenc}

\usepackage{t1enc}

\usepackage[mathscr]{eucal}

\usepackage{indentfirst}

\usepackage{graphicx}

\usepackage{graphics}

\usepackage{pict2e}

\usepackage{mathrsfs}

\usepackage{enumerate}
\usepackage{ferrers} 
\usepackage[pagebackref]{hyperref}
\hypersetup{colorlinks=true}
\usepackage{cite}
\usepackage{color}
\usepackage{epic}
\usepackage{hyperref} 
\usepackage{framed}
\usepackage{mathabx}

\numberwithin{equation}{section}
\theoremstyle{plain}

\newtheorem{theorem}{Theorem}[section]

\newtheorem{lemma}[theorem]{Lemma}

\newtheorem{corollary}[theorem]{Corollary}

\newtheorem{proposition}[theorem]{Proposition}

\newtheorem{Fact}[theorem]{Fact}

\theoremstyle{definition}

\newtheorem{Def}[theorem]{Definition}

\newtheorem{example}[theorem]{Example}

\newtheorem{remark}[theorem]{Remark}

\newtheorem{?}[theorem]{Problem}

\newtheoremstyle{named}{}{}{\itshape}{}{\bfseries}{.}{.5em}{#1\thmnote{ #3}}
\theoremstyle{named}

\newcommand{\abs}[1]{\left|#1\right|}
\newcommand{\set}[1]{\left\{#1\right\}}
\newcommand{\qbin}[3]{\left[\begin{matrix} #1 \\ #2 \end{matrix} \right]_{#3}}

\newcommand{\f}[1]{\ifthenelse{\equal{#1}{1}}{(q;q)_\infty}{(q^{#1};q^{#1})_{\infty}}}

\def\cD{\mathcal{D}}
\def\cP{\mathcal{P}}
\def\cA{\mathcal{A}}
\def\Ao{\mathcal{A}^{\mathrm{I}}}
\def\A2{\mathcal{A}^{\mathrm{II}}}
\def\cB{\mathcal{B}}
\def\cC{\mathcal{C}}
\def\cO{\mathcal{O}}
\def\bN{\mathbb{N}}
\def\la{\lambda}
\def\sol{\mathrm{sol}}
\def\Dur{\mathrm{Dur}}
\def\dur{\mathrm{dur}}
\def\alt{\mathrm{alt}}
\def\pai{\mathrm{pi}}
\def\wt{\mathrm{wt}}

\begin{document}
\title[SOL in strict partitions II: the $2$-measure and refinements of Euler's theorem]{Sequences of odd length in strict partitions II: the $2$-measure and refinements of Euler's theorem}

\author[S. Fu]{Shishuo Fu}
\address[Shishuo Fu]{College of Mathematics and Statistics, Chongqing University, Chongqing 401331, P.R. China}
\email{fsshuo@cqu.edu.cn}

\author[H. Li]{Haijun Li}
\address[Haijun Li]{College of Mathematics and Statistics, Chongqing University, Chongqing 401331, P.R. China}
\email{lihaijun@cqu.edu.cn}

\date{\today}

\begin{abstract}
The number of sequences of odd length in strict partitions (denoted as $\sol$), which plays a pivotal role in the first paper of this series, is investigated in different contexts, both new and old. Namely, we first note a direct link between $\sol$ and the $2$-measure of strict partitions when the partition length is given. This notion of $2$-measure of a partition was introduced quite recently by Andrews, Bhattacharjee, and Dastidar. We establish a $q$-series identity in three ways, one of them features a Franklin-type involuion. Secondly, still with this new partition statistic $\sol$ in mind, we revisit Euler's partition theorem through the lens of Sylvester-Bessenrodt. Two new bivariate refinements of Euler's theorem are established, which involve notions such as MacMahon's 2-modular Ferrers diagram, the Durfee side of partitions, and certain alternating index of partitions that we believe is introduced here for the first time.
\end{abstract}


\maketitle

\begin{flushright}
{\noindent\footnotesize\itshape Dedicated to George Andrews and Bruce Berndt for their 85th birthdays.}
\end{flushright}

\bigskip\bigskip

\section{Introduction}\label{sec:intro}
In a previous paper~\cite{FL24}, we enumerated strict partitions (i.e., partitions without repeated parts) $\la$ with respect to the size $\abs{\la}$, the length (i.e., the number of parts) $\ell(\la)$, and the number of sequences of odd length $\sol(\la)$. Denote the set of strict partitions by $\cD$, then one of the main results in~\cite{FL24} can be written in terms of generating function as follows:
\begin{align}\label{gf:l-sol}
D^{\sol,\ell}(x,y;q):=\sum_{\la\in\cD}x^{\sol(\la)}y^{\ell(\la)}q^{\abs{\la}}=\sum_{i,j\ge 0}\frac{x^iy^{i+2j}q^{i^2+2ij+2j^2+j}}{(q;q)_i(q^2;q^2)_j}.
\end{align}
Here and in the sequel, we adopt the customary $q$-notations, and mostly follow notations in \cite{FL24} or \cite{andtp} unless otherwise noted. For $\abs{q}<1$, we let
\begin{align*}
& (a;q)_n=(1-a)(1-aq)\cdots(1-aq^{n-1}), \text{ for $n\ge 1$,}\\
& (a;q)_0=1,\text{ and } (a;q)_{\infty}=\lim_{n\to\infty}(a;q)_n.
\end{align*}

In this follow-up paper, we continue our study on this relatively new partition statistic ``$\sol$'', connecting it with the $2$-measure of partitions, a notion introduced by Andrews, Bhattacharjee, and Dastidar in \cite{ABD22}. This link immediately gives rise to the following identity, which could be viewed as our first main result.

\begin{theorem}\label{thm:double sum vs Andrews}
Let $x,y,q$ be complex numbers and $\abs{q}<1$. Then we have
\begin{align}\label{id:double sum vs Andrews}
\sum_{i,j\ge 0}\frac{x^{i+j}y^{i+2j}q^{i^2+2ij+2j^2+j}}{(q;q)_i(q^2;q^2)_j} &= (-yq;q)_{\infty}\sum_{n\ge 0}\frac{(-1)^ny^nq^n(x;q^2)_n}{(q;q)_n}.
\end{align}
\end{theorem}

Moreover, we establish the following two bivariate refinements of Euler's famed ``odd-distinct'' partition theorem. The undefined terms will be explained in section~\ref{sec:pre}.

\begin{theorem}\label{thm:ref-d}
For integers $n\ge k\ge 1$ and $m\ge 0$, let $\cD_{n,k,m}$ (resp.~$D(n,k,m)$) denote the set (resp.~number) of partitions of $n$ into $k$ distinct parts wherein there are exactly $m$ sequences of odd length. Let 
$\Ao_{n,k,m}$ (resp.~$A_1(n,k,m)$) denote the set (resp.~number) of partitions of $n$ into odd parts whose $2$-modular diagram is of type I with Durfee side being $k$ and $2$-modular sub-Durfee side being $m$. Let $\A2_{n,k,m}$ (resp.~$A_2(n,k,m)$) denote the set (resp.~number) of partitions of $n$ into odd parts whose $2$-modular diagram is of type II with Durfee side being $k$ and $2$-modular sub-Durfee side being $m$. Then
\begin{align*}
A_1(n,k,m) &= D(n,2k,2m),\\
A_2(n,k,m) &= D(n,2k-1,2m+1).
\end{align*}
\end{theorem}

\begin{theorem}\label{thm:ref-alt}
Let $\cB_{n,k,m}$ (resp.~$B(n,k,m)$) denote the set (resp.~number) of partitions of $n$ into odd parts whose $2$-modular Durfee side is $k$ and alternating index is $m$. Then
\begin{align*}
B(n,\lceil k/2\rceil,m) &= D(n,k,m).
\end{align*}
\end{theorem}

It is easily seen from definition that the two parameters $k$ and $m$ in $\cD_{n,k,m}$ have the same parity, so we have the identity $\bigcup_{k\ge 1,m\ge 0}\cD_{n,k,m}=\bigcup_{k\ge 1,m\ge 0}(\cD_{n,2k,2m}\bigcup\cD_{n,2k-1,2m+1})$. Therefore, summing the identities in either theorem over $m\in[0,k]$ results in the following refinement of Euler's theroem, which is readily implied from Sylvester's bijective proof of Euler's theorem; see \cite[Propositions 2.2]{bes94} and \cite[Theorem 1]{zen05}.
\begin{corollary}[Sylvester-Bessenrodt]
The number of partitions of $n$ into $k$ distinct parts is equal to the number of partitions of $n$ into odd parts whose $2$-modular Durfee side is $\lceil k/2\rceil$.
\end{corollary}
\begin{example}
For $(n,k,m)=(16,4,2)$, the three associated sets of restricted partitions are listed below and each of them contains six partitions, as anticipated by theorems~\ref{thm:ref-d} and \ref{thm:ref-alt}. We abbreviate repeated parts using superscript, so $3^2$ refers to $3+3$.
\begin{align*}
\Ao_{16,2,1} &= \set{5^2+3+1^3,5^2+1^6,7+5+3+1,7+5+1^4,9+5+1^2,7^2+1^2},\\
\cB_{16,2,2} &= \set{5^2+3^2,5^2+3+1^3,5^2+1^6,7+5+1^4,9+5+1^2,7^2+1^2},\\
\cD_{16,4,2} &= \set{10+3+2+1, 9+4+2+1, 8+5+2+1, 8+4+3+1, 7+4+3+2, 6+5+4+1}.
\end{align*}

For $(n,k,m)=(15,3,1)$, the three associated sets of restricted partitions are listed below, each contains five partitions.
\begin{align*}
\A2_{15,2,0} &= \set{11+3+1, 9+3+1^3, 7+3+1^5, 5+3+1^7, 3^2+1^9},\\
\cB_{15,2,1} &= \set{9+3^2, 3^5, 3^4+1^3, 3^3+1^6, 3^2+1^9},\\
\cD_{15,3,1} &= \set{12+2+1, 10+3+2, 8+4+3, 7+6+2, 6+5+4}.
\end{align*}
\end{example}

We organize the rest of the paper as follows. Definitions and some preliminary results are collected in section \ref{sec:pre}. Then we present three proofs of \eqref{id:double sum vs Andrews} in section~\ref{sec:3pf}, one of which is based on the connection between our new statistic $\sol(\lambda)$ and the $2$-measure of $\lambda$ (see~\eqref{id:2measure-sol}). Next in section \ref{sec:pf thm1}, we prove theorem \ref{thm:ref-d} via generating function manipunation, while two proofs of theorem \ref{thm:ref-alt} are given in section \ref{sec:pf thm2}, one through a closer look at Sylvester's bijection, the other uses again generating function.

\section{Preliminaries}\label{sec:pre}

Ferrers diagram~\cite[Chap.~1.3]{andtp} is a commonly used pictorial tool to illustrate a partition. To each partition $\lambda \vdash n$ is associated its {\it Ferrers diagram} denoted as $[\lambda]$, which is the set of $n$ left-aligned unit cells such that there are precisely $\lambda_i$ cells in the $i$-th row, for $1\le i\le \ell(\lambda)$; see Fig.~\ref{fig:Ferrers} for the Ferrers diagrams of two concrete partitions (neglect the letters D and d, they are for future reference). Besides the clear advantage of being intuitive, it is quite often the case that the representation of partitions using Ferrers diagram brings new combinatorial insights on the investigation of various partition problems. We introduce the following notions related to partitions, some of which are better perceived via their associated Ferrers diagrams.

\begin{figure}[ht]
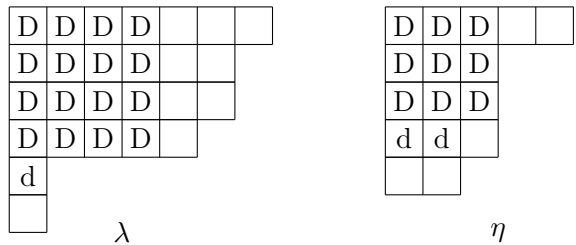

\begin{ferrers}
\addcellrows{7+6+6+5+1+1}
\highlightcellbyletter{1}{1}{D}
\highlightcellbyletter{1}{2}{D}
\highlightcellbyletter{1}{3}{D}
\highlightcellbyletter{1}{4}{D}
\highlightcellbyletter{2}{1}{D}
\highlightcellbyletter{2}{2}{D}
\highlightcellbyletter{2}{3}{D}
\highlightcellbyletter{2}{4}{D}
\highlightcellbyletter{3}{1}{D}
\highlightcellbyletter{3}{2}{D}
\highlightcellbyletter{3}{3}{D}
\highlightcellbyletter{3}{4}{D}
\highlightcellbyletter{4}{1}{D}
\highlightcellbyletter{4}{2}{D}
\highlightcellbyletter{4}{3}{D}
\highlightcellbyletter{4}{4}{D}
\highlightcellbyletter{5}{1}{d}
\addtext{1.5}{-3}{$\la$}
\putright
\addcellrows{5+3+3+3+2}
\highlightcellbyletter{1}{1}{D}
\highlightcellbyletter{1}{2}{D}
\highlightcellbyletter{1}{3}{D}
\highlightcellbyletter{2}{1}{D}
\highlightcellbyletter{2}{2}{D}
\highlightcellbyletter{2}{3}{D}
\highlightcellbyletter{3}{1}{D}
\highlightcellbyletter{3}{2}{D}
\highlightcellbyletter{3}{3}{D}
\highlightcellbyletter{4}{1}{d}
\highlightcellbyletter{4}{2}{d}
\addtext{1.5}{-3}{$\eta$}
\end{ferrers}
\caption{The Ferrers diagrams of partitions $\lambda=7+6^2+5+1^2$ and $\eta=5+3^3+2$}
\label{fig:Ferrers}
\end{figure}

\begin{Def}[{cf.~\cite[Chap.~2.3]{andtp}}]\label{def:Durfee}
For each partition $\la$, the largest square, say a $k\times k$ square, that fits into its Ferrers diagram (including the upper-leftmost cell) is called the {\it Durfee square} of $\la$ (or of $[\la]$), and the number $k$ is called the {\it Durfee side} of $\la$. Alternatively, $k$ is the largest integer $i$ such that $\lambda_i$ is at least $i$. We agree that the Durfee side of the empty partition $\epsilon$ is zero.
\end{Def}

Given a partition $\la$ with Durfee side $k$, we can express $\la$ alternatively as the triple $(k;\alpha,\beta)$, where $\alpha$ and $\beta$ are the subpartitions to the right of and below the Durfee square of $\la$, respectively. Observe that $\abs{\lambda}=k^2+\abs{\alpha}+\abs{\beta}$. For instance, the first partition in Fig.~\ref{fig:Ferrers} can thus be written as $(4;3+2^2+1,1^2)$. Such a triple notation is preferred for the following definition.

\begin{Def}\label{def:subDurfee}
For a partition written as the triple $\la=(k;\alpha,\beta)$, we define its {\it sub-Durfee side} according to the following two cases:
\begin{enumerate}
	\item If $\la_k>k$, we say $\la$ is of type I, and its sub-Durfee side is taken to be the Durfee side of $\beta$.
	\item If $\la_k=k$, we say $\la$ is of type II, and its sub-Durfee side is taken to be the largest integer $j$, such that the $j$ (vertically) by $j+1$ (horizontally) rectangle can fit into $[\beta]$. When no such $j$ exists, like the cases when $\beta$ is empty or $\beta$ only consists of parts of size one, we consider the sub-Durfee side of $\la$ to be zero.
\end{enumerate}
\end{Def}

For better illustation, the two Ferrers diagrams in Fig.~\ref{fig:Ferrers} have been filled in with letters D and d to indicate their Durfee sides and sub-Durfee sides, respectively. Note that $\la$ is of type I while $\eta$ is of type II.

\begin{remark}
Two remarks concerning this newly introduced notion of sub-Durfee side are in order. Firstly, the split between two subcases hinges on whether or not we have $\la_k=k$. This is reminiscent of Andrews's arithmetical approach \cite[Sect.~4]{and72} to the Rogers-Fine identity; see also Zeng's related work~\cite[Sect.~3]{zen05}. Moreover, the subtle discrepancy between two cases in the definition of sub-Durfee side makes it different from, albeit relevent to, the notion of ``successive Durfee square'' considered by Andrews~\cite{and79}; see also \cite[Sect.~3.1]{war97}.
\end{remark}

It was MacMahon~\cite{mac23} who generalized the Ferrers diagram to an $M$-modular setting. In particular, the use of $2$-modular diagram proves to be convenient in analyzing Sylvester's bijective proof of Euler's theorem~\cite{bes94}, and in dealing with partitions with odd parts distinct~\cite{LO09}. We require $2$-modular diagram in the definitions of both sets $\cA_{n,k,m}$ and $\cB_{n,k,m}$, as well as the proofs of theorems~\ref{thm:ref-d} and \ref{thm:ref-alt}.

\begin{Def}
For any integer $M\ge 1$, an $M$-modular (Ferrers) diagram is a Ferrers diagram where all the cells are filled with $M$'s except possibly the last cell of a row, which may be filled with any number from $\set{1,2,\ldots,M}$, with the condition that the numbers along each column are weakly decreasing from top to bottom. Given a partition $\la$, there exists a unique $M$-modular diagram, which we denote as $[\la]_M$, such that the numbers in the $i$-th row sum up to $\la_i$ for $1\le i\le \ell(\la)$.
\end{Def}

As an illustration, the partition $\lambda=7+6^2+5+1^2$ has its associated $2$-modular diagram depicted below.

\begin{ferrers}
\addcellrows{4+3+3+3+1+1}
\highlightcellbyletter{1}{1}{2}
\highlightcellbyletter{1}{2}{2}
\highlightcellbyletter{1}{3}{2}
\highlightcellbyletter{1}{4}{1}
\highlightcellbyletter{2}{1}{2}
\highlightcellbyletter{2}{2}{2}
\highlightcellbyletter{2}{3}{2}
\highlightcellbyletter{3}{1}{2}
\highlightcellbyletter{3}{2}{2}
\highlightcellbyletter{3}{3}{2}
\highlightcellbyletter{4}{1}{2}
\highlightcellbyletter{4}{2}{2}
\highlightcellbyletter{4}{3}{1}
\highlightcellbyletter{5}{1}{1}
\highlightcellbyletter{6}{1}{1}
\addtext{1.5}{-3}{$[\la]_2$}
\end{ferrers}

What prompted us to derive \eqref{id:double sum vs Andrews} was an observation (see~\eqref{id:2measure-sol} below) that links $\sol(\la)$ to the so-called $2$-measure of $\la$. More generally, let us first define what is the $k$-measure of a partition.

\begin{Def}[{cf.~\cite[Defn.~1]{ABD22}}]\label{def:k-measure}
For a positive integer $k$ and a partition $\la$, the {\it $k$-measure} of $\la$ is the length of the longest subsequence (not necessarily consecutive) of parts of $\la$ wherein the difference between any two parts of the subsequence is at least $k$.
\end{Def}

\noindent \textbf{Some Notations}. For the reader's convenience, we collect and recall here most of the notations used in this paper, some of them have already been mentioned in the introduction or appeared in \cite{FL24}. Given any partition $\la$, we let
\begin{itemize}
\item $\cP$ and $\cD$ denote the set of partitions and the set of partitions into distinct parts, respectively;

\item $|\la|$ denote the size of $\la$, that is, the sum of all parts of $\la$;

\item $\ell(\la)$ denote the length of $\la$, that is, the number of parts of $\la$;

\item $\Dur(\la)$ denote the Durfee side of $[\la]$;

\item $\dur(\la)$ denote the sub-Durfee side of $[\la]$;

\item $\Dur_2(\la)$ denote the Durfee side of $[\la]_2$, and we refer to it as the $2$-modular Durfee side of $\la$;

\item $\dur_2(\la)$ denote the sub-Durfee side of $[\la]_2$, and we refer to it as the $2$-modular sub-Durfee side of $\la$;

\item $\mu_{k}(\la)$ denote the $k$-measure of $\la$ for a positive integer $k$.

\item If $\la\in \cD$, we denote $\sol(\la)$ the number of sequences of odd length in $\la$.
\end{itemize}

For our running example $\la=7+6^2+5+1^2$ from Fig.~\ref{fig:Ferrers}, we have $\abs{\la}=26$, $\ell(\la)=6$, $\Dur(\la)=4$, $\dur(\la)=1$, $\Dur_2(\la)=3$, $\dur_2(\la)=1$, and $\mu_2(\la)=3$. If $\sigma=7+6+5+2+1\in\cD$, then we see $\sol(\sigma)=1$ since it contains only one sequence of odd length, namely $(7,6,5)$.

It is easy to see that the $1$-measure of a partition is essentially the number of distinct parts it contains. While for $2$-measure, we make the following key observation.

\begin{proposition}
Given any partition $\la\in \cD$, we have
\begin{align}\label{id:2measure-sol}
2\mu_{2}(\la)=\ell(\la)+\sol(\la).
\end{align}
\end{proposition}

\begin{proof}
Firstly, for $m\in \bN^+$ suppose that
\begin{align*}
\la=(\la_{1}+\la_{2}+\cdots +\la_{k_{1}})+(\la_{k_{1}+1}+\cdots +\la_{k_{2}})+\cdots +(\la_{k_{m-1}+1}+\cdots +\la_{k_{m}})\in \cD,
\end{align*}
where $k_m=\ell(\la)$, $\la_{k_{i}}-\la_{k_{i}+1}\ge 2$ for $i=1, 2, \ldots, m-1$, and parts inside the same parentheses are consecutive. Furthermore, by the definition of $2$-measure we have (setting $k_0:=0$)
\begin{align*}
2\mu_{2}(\la)&=2\sum_{i=1}^{m}\left(\sum_{\scriptscriptstyle k_{i}-k_{i-1}\text{ odd}}\frac{k_{i}-k_{i-1}+1}{2}+\sum_{\scriptscriptstyle k_{i}-k_{i-1}\text{ even}}\frac{k_{i}-k_{i-1}}{2}\right)\\
&=\sum_{i=1}^{m}\left(\sum_{\scriptscriptstyle k_{i}-k_{i-1}\text{ odd}}(k_{i}-k_{i-1}+1)+\sum_{\scriptscriptstyle k_{i}-k_{i-1}\text{ even}}(k_{i}-k_{i-1})\right)\\
&=\sum_{i=1}^{m}(k_{i}-k_{i-1})+\sum_{i=1}^{m}\sum\limits_{\scriptscriptstyle\text{$k_i-k_{i-1}$ odd}} 1\\
&=\ell(\la)+\sol(\la).
\end{align*}
\end{proof}

\begin{example}
Given $\la=14+13+11+9+6+5+4+2+1\in \cD$, one verifies that $\sol(\la)=3$, $\ell(\la)=9$ and $\mu_{2}(\la)=6$, so $2\mu_{2}(\la)=\ell(\la)+\sol(\la)$ indeed.
\end{example}

We end this section with the following notion of ``union of partitions'', which will be needed for later constructions.
\begin{Def}\label{def:union}
The {\it union} of two partitions $\la$ and $\mu$, denoted as $\la\bigcup\mu$, is the partition whose parts form the set theoretical union of the parts of $\la$ and $\mu$. So for example, given $\la=3+1+1$ and $\mu=4+3+2+1$, we have $\la\bigcup\mu=4+3^2+2+1^3$.
\end{Def}

\section{\texorpdfstring{Three proofs of theorem~\ref{thm:double sum vs Andrews}}{Three proofs of Theorem 1.1}}\label{sec:3pf}
In a follow-up paper to \cite{ABD22}, Andrews, Chern, and Li~\cite{ACL22} deduced the following trivariate generating function of strict partitions that holds for any $k\ge 1$:
\begin{align}\label{gf:k-measure}
\sum_{\la\in\cD}x^{\mu_{k}(\la)}y^{\ell(\la)}q^{|\la|} &=(-yq; q)_{\infty}\sum_{n\geq 0}\frac{(-1)^ny^nq^n(x; q^{k})_{n}}{(q; q)_{n}}.
\end{align}
With the relation \eqref{id:2measure-sol} in mind, we immediately get the following proof of \eqref{id:double sum vs Andrews}.

\begin{proof}[1st proof of theorem~\ref{thm:double sum vs Andrews}]
In view of \eqref{id:2measure-sol}, we make the change of variables $x\to x^{\frac{1}{2}}$ and $y\to x^{\frac{1}{2}}y$ in \eqref{gf:l-sol} to get
\begin{align}\label{id:2-measure-pos}
\sum_{\la\in\cD}x^{\mu_2(\la)}y^{\ell(\la)}q^{\abs{\la}}=\sum_{i,j\ge 0}\frac{x^{i+j}y^{i+2j}q^{i^2+2ij+2j^2+j}}{(q;q)_i(q^2;q^2)_j}.
\end{align}
This combined with \eqref{gf:k-measure} in the case of $k=2$ establishes \eqref{id:double sum vs Andrews}.
\end{proof}

The relation \eqref{id:2measure-sol} and Andrews-Chern-Li's identity \eqref{gf:k-measure} were what led us to discover theorem~\ref{thm:double sum vs Andrews}. But in retrospect, we are able to give two direct proofs of \eqref{id:double sum vs Andrews} without going through \eqref{gf:k-measure}. One proof is via standard $q$-series manipunation and we thank Krattenthaler~\cite{kra24} for sharing it with us. We need Cauchy's $q$-binomial theorem and the following version of $q$-Chu-Vandermonde summation formula.
\begin{proposition}[{\cite[Eq.~(1.3.2); Appendix (II.3)]{GR90}}]
For $\abs{x}<1$ and $\abs{q}<1$, we have
\begin{align}\label{id:qbinom}
_{1}\phi_{0}(a; -; q, x)=\sum_{m\geq 0}\frac{(a; q)_{m}}{(q; q)_{m}}x^m=\frac{(ax; q)_{\infty}}{(x; q)_{\infty}}.
\end{align}
\end{proposition}
\begin{proposition}[{\cite[Eq.~(1.5.3); Appendix (II.6)]{GR90}}]
\begin{align}\label{id:qChuVan}
_{2}\phi_{1}(a, q^{-N}; c; q, q)=\frac{a^N(c/a; q)_{N}}{(c; q)_{N}},
\end{align}
where $N$ is a non-negative integer and we have used the standard basic hypergeometric series notation
$$_2\phi_1(a,b;c;q,x):=\sum_{n\ge 0}\frac{(a;q)_n(b;q)_n}{(q;q)_n(c;q)_n}x^n.$$
\end{proposition}

\begin{proof}[2nd proof of theorem~\ref{thm:double sum vs Andrews}]
One application of $q$-binomial theorem \eqref{id:qbinom} gives us
\begin{align*}
_{1}\phi_{0}(q^{-n}; -; q, x)=\sum_{i\geq 0}\frac{(q^{-n}; q)_{i}}{(q; q)_{i}}x^i=(xq^{-n}; q)_{n}.
\end{align*}
Let $x\rightarrow xq^n$ and then let $q\rightarrow q^2$ we can obtain that
\begin{align*}
(x; q^2)_{n}=\sum_{i\geq 0}\frac{(q^{-2n}; q^2)_{i}}{(q^2; q^2)_{i}}(xq^{2n})^{i}.
\end{align*}
Moreover note that for $0\le i\le n$,
\begin{align*}
(q^{-2n}; q^2)_{i}&=(1-q^{-2n})(1-q^{-2n+2})\cdots (1-q^{-2n+2i-2})\\
&=(-1)^{i}q^{-2ni+i^2-i}(1-q^{2n})(1-q^{2n-2})\cdots (1-q^{2n-2i+2})\\
&=(-1)^iq^{-2ni+i^2-i}\frac{(q^2; q^2)_{n}}{(q^2; q^2)_{n-i}}.
\end{align*}
Then we have 
\begin{equation}
(x; q^2)_{n}=\sum_{i\geq 0}(-x)^{i}q^{i^2-i}\frac{(q^2; q^2)_{n}}{(q^2; q^2)_{i}(q^2; q^2)_{n-i}}.\label{2}
\end{equation}
Next recall the $q$-exponential function (which is also a special case of \eqref{id:qbinom})
\begin{align*}
E_{q}(z)=\sum_{n\geq 0}\frac{q^{\binom{n}{2}}z^n}{(q; q)_{n}}=(-z; q)_{\infty}.
\end{align*}
Hence,
\begin{equation}
(-yq; q)_{\infty}=\sum_{m\geq 0}\frac{q^{\binom{m}{2}}(yq)^m}{(q; q)_{m}}.\label{3}
\end{equation}
We plug \eqref{2} and \eqref{3} to the right hand side of \eqref{id:double sum vs Andrews} to get 
\begin{align*}
\sum_{m\geq 0}\frac{q^{\binom{m}{2}}(yq)^m}{(q; q)_{m}}\sum_{n\geq 0}\frac{(-1)^ny^nq^n}{(q; q)_{n}}\sum_{i= 0}^{n}(-x)^{i}q^{i^2-i}\frac{(q^2; q^2)_{n}}{(q^2; q^2)_{i}(q^2; q^2)_{n-i}}.
\end{align*}
Setting $m+n=j$ yields
\begin{align*}
\sum_{j\geq 0}\frac{q^{\binom{j-n}{2}}(yq)^{j-n}}{(q; q)_{j-n}}\sum_{n=0}^{j}\frac{(-1)^ny^nq^n}{(q; q)_{n}}\sum_{i= 0}^{n}(-x)^{i}q^{i^2-i}\frac{(q^2; q^2)_{n}}{(q^2; q^2)_{i}(q^2; q^2)_{n-i}}.
\end{align*}
After some simplification, we have
\begin{align*}
\sum_{j\geq 0}\sum_{n=0}^{j}\sum_{i=0}^{n}y^{j}x^{i}\frac{q^{j+i^2-i}}{(q^2; q^2)_{i}}\frac{(-1)^{n+i}q^{\binom{j-n}{2}}(-q; q)_{n}}{(q; q)_{j-n}(q^2; q^2)_{n-i}}.
\end{align*}
Exchange summation order, we have
\begin{equation}
\sum_{j\ge i\ge 0}y^{j}x^{i}\frac{q^{j+i^2-i}}{(q^2; q^2)_{i}}\sum_{i\le n\le j}\frac{(-1)^{n+i}q^{\binom{j-n}{2}}(-q; q)_{n}}{(q; q)_{j-n}(q^2; q^2)_{n-i}}.\label{5}
\end{equation}
We know that when $i\le n\le j$,
\begin{align}
& (q; q)_{j-i}=(q; q)_{j-n}(q^{j-n+1}; q)_{n-i},\label{6}\\
& (-1)^{n-i}q^{\binom{j-n}{2}}(q^{j-n+1}; q)_{n-i}=q^{\binom{j-i}{2}+(n-i)}(q^{-j+i}; q)_{n-i},\label{7}
\end{align}
and
\begin{align}
(-q; q)_{n} &=(-q; q)_{i}(-q^{i+1}; q)_{n-i}.\label{8}
\end{align}
Apply equations \eqref{6}, \eqref{7} and \eqref{8} to the expression \eqref{5} to derive
\begin{align}
&\sum_{j\ge i\ge 0}y^{j}x^{i}\frac{q^{j+i^2-i}}{(q^2; q^2)_{i}}\frac{q^{\binom{j-i}{2}}(-q; q)_{i}}{(q; q)_{j-i}}\sum_{i\le n\le j}\frac{q^{n-i}(-q^{i+1}; q)_{n-i}(q^{-j+i}; q)_{n-i}}{(q; q)_{n-i}(-q; q)_{n-i}}\nonumber\\
&=\sum_{j\ge i\ge 0}y^jx^i\frac{q^{\binom{j-i}{2}+j+i^2-i}}{(q; q)_{i}(q; q)_{j-i}} {}_{2}\phi_{1}(-q^{i+1}, q^{-j+i}; -q; q, q)\nonumber\\
&=\sum_{i,j \ge 0}y^{i+j}x^i\frac{q^{\binom{j+1}{2}+i^2}}{(q; q)_{i}(q; q)_{j}} {}_{2}\phi_{1}(-q^{i+1}, q^{-j}; -q; q, q) \label{10}
\end{align}
Putting $(a, N, c)\to (-q^{i+1}, j, -q)$ in \eqref{id:qChuVan} we get 
\begin{align}\label{12}
_{2}\phi_{1}(-q^{i+1}, q^{-j}; -q; q, q)=\frac{(-q^{i+1})^{j}(q^{-i}; q)_{j}}{(-q; q)_{j}}.
\end{align}
Here the factor $(q^{-i};q)_j$ indicates that we shall require $j\le i$ to get non-vanishing terms. Hence we see that
\begin{equation}
(-1)^{j}\frac{(q^{-i}; q)_{j}}{(q; q)_{i}}=\frac{q^{\binom{j}{2}-ij}}{(q; q)_{i-j}}.\label{13}
\end{equation}
Plug \eqref{12} and \eqref{13} back to \eqref{10} and simplify to obtain
\begin{align*}
\sum_{i, j\geq 0}y^{i+j}x^{i}\frac{q^{i^2+j^2+j}}{(q; q)_{i-j}(q^2; q^2)_{j}}.
\end{align*}
Let $i\to i+j$ to arrive at the left hand side of \eqref{id:double sum vs Andrews}.
\end{proof}

Our final proof of theorem \ref{thm:double sum vs Andrews} is purely combinatorial. Namely, we first interpret the right hand side of \eqref{id:double sum vs Andrews} as the generating function of certain signed pairs of labeled partitions. Then we construct a sign reversing, weight preserving involution, such that the fixed points under this involution turn out to be generated by the left hand side of \eqref{id:double sum vs Andrews}.

We begin by rewritting the right hand side of \eqref{id:double sum vs Andrews} as follows.
\begin{align*}
(-yq; q)_{\infty}\sum_{n\geq 0}\frac{(-1)^ny^nq^n(x; q^{2})_{n}}{(q; q)_{n}}
&=(-yq; q)_{\infty}\sum_{n\geq 0}\frac{q^n(xy-y)(xyq^2-y)\cdots (xyq^{2n-2}-y)}{(q; q)_{n}}.
\end{align*}
Now we see $(-yq; q)_{\infty}$ generates partitions, say $\lambda\in\cD$, wherein each part receives a label $y$. The summation, on the other hand, generates the set of partitions that we denote as $\cC$, whose member, say $\eta\in\cC$, satisfies the following conditions.
\begin{enumerate}[(i)]
\item Each part is labeled as either $xy$ or $y$. If $\eta_{i}$ is labeled as $xy$, then $\eta_{i-1}-\eta_{i}\geq 2$. 

\item Set $\eta_{0}=\infty$ to make sure $\eta_{1}$ can be labeled as either $xy$ or $y$. 

\item Suppose there are $m$ parts of $\eta$ that are labeled as $y$, then the sign of $\eta$ is taken to be $(-1)^m$. 
\end{enumerate}

Let us take $(3+2,6_x+3+3+1_x)\in\cD\times\cC$ for instance, where the subscript $x$ indicates this part is labeled as $xy$ and all remaining parts are labeled as $y$. This pair has weight $x^2y^6q^{18}$ with sign $(-1)^3=-1$. In general, the power of $x$ records the number of parts with label $xy$ while the powers of $y$ and $q$ are given by $\ell(\la)+\ell(\eta)$ and $\abs{\la}+\abs{\eta}$, respectively. We denote the weight of $(\la,\eta)$ as $\wt(\la,\eta)$.

We are going to construct an involution $\varphi$ on $\cD\times\cC$ such that the partition pair $(\la, \eta)\in (\cD\times\cC)\setminus (\cD'\times\cC')$ has the same weight and opposite sign with $\varphi((\la, \eta))$, where $\cD'\times\cC'$ is the set of fixed points under $\varphi$. 

Given any pair $(\la,\eta)\in\cD\times\cC$, find the smallest part, say $a$, of $\la$ such that $\eta$ does not contain $(a-1)_{x}$ as a part. Let $b$ be the smallest part of $\eta$ that is labeled as $y$. If there is no such $a$ or $b$, then take $a=\infty$ or $b=\infty$. There are three cases to consider. In cases (1) and (2), we take $\varphi((\la,\eta))=(\hat{\la},\hat{\eta})$.


\begin{enumerate}
\item If $a>b$, then delete $b$ from $\eta$ and insert $b$ into $\la$ as a new part. This outputs a new pair that we denote as $(\hat{\la}, \hat{\eta})$.

\item If $a\leq b$, then delete $a$ from $\la$ and insert $a$ into $\eta$ as a new part with label $y$. The new pair thus derived is denoted as $(\hat{\la}, \hat{\eta})$.

\item If $a=b=\infty$, then we treat $(\la,\eta)$ as a fixed point, thus $\varphi((\la,\eta))=(\la,\eta)$.
\end{enumerate}


\begin{proof}[3rd proof of theorem~\ref{thm:double sum vs Andrews}]
It suffices to characterize $\cD'\times\cC'$ and show that $\varphi$ as defined above is indeed a sign reversing, weight preserving involution. 

For the first goal, note that if $(\la,\eta)\in\cD'\times\cC'$, then $a=\infty$, which means $\eta$ must contain part $(\la_i-1)_x$ for each part $\la_i$ of $\la$. And all parts of $\eta$ are labeled as $xy$ since $b=\infty$ (in particular, $(\la,\eta)$ has a positive sign). Moreover, the parts of $\la$ are dijoint from parts of $\eta$. With all these being noticed, it should be clear that taking the union (recall definition~\ref{def:union})
\begin{align*}
\cD'\times\cC' &\to \cD\\
(\la,\eta) &\mapsto \tau:=\la\bigcup\eta
\end{align*}
is a bijection\footnote{Since parts from $\la$ and parts from $\eta$ have different labels, we can uniquely recover the pair $(\la,\eta)$ from the union $\la\bigcup\eta$.} between $\cD'\times\cC'$ and $\cD$, such that the sequences in $\tau$ consist of parts with alternating labels $y$ and $xy$, and every sequence ends with a part labeled $xy$. This alternating feature is nicely captured by the $2$-measure. In summary, we have
\begin{align*}
\sum_{(\la,\eta)\in\cD'\times\cC'}\wt(\la,\eta)=\sum_{\tau\in\cD} x^{\mu_2(\tau)}y^{\ell(\tau)}q^{\abs{\tau}}.
\end{align*}
This is exactly the left hand side of \eqref{id:double sum vs Andrews} in view of \eqref{id:2-measure-pos}.

Next, to show that $\varphi$ meets our requirements, we leave the verification of the following items to the reader.
\begin{itemize}
	\item In cases (1) and (2), $\varphi$ is well-defined. I.e., $(\hat{\la},\hat{\eta})$ belongs to $\cD\times\cC$ in both cases.
	\item If $(\la,\eta)$ is in case (1), then its image $(\hat{\la},\hat{\eta})$ is in case (2); if $(\la,\eta)$ is in case (2), then its image $(\hat{\la},\hat{\eta})$ is in case (1).
	\item For cases (1) and (2), $(\la,\eta)$ and $(\hat{\la},\hat{\eta})$ have the same weight but opposite sign.
	\item We have $\varphi^2((\la,\eta))=(\la,\eta)$ for all pairs $(\la,\eta)\in\cD\times\cC$.
\end{itemize}
\end{proof}

\begin{remark}
Our third proof of theorem~\ref{thm:double sum vs Andrews} could also be viewed as a combinatorial proof of \eqref{gf:k-measure} in the case of $k=2$. In the same vein, for another identity (over the set $\cP$ of all partitions rather than $\cD$) from the same paper of Andrews-Chern-Li~\cite[$k=2$ case of Eq.~(1)]{ACL22} 
\begin{align}\label{id:2measure-P}
\sum_{\la\in \cP}x^{\mu_{2}(\la)}y^{\ell(\la)}q^{|\la|}=\frac{1}{(yq; q)_{\infty}}\sum_{n\geq 0}\frac{(-1)^ny^nq^{n(n+1)/2}(x; q)_{n}}{(q; q)_n},
\end{align}
one can also construct an involution so as to explain the cancellation happened on the right hand side of \eqref{id:2measure-P}, and then recognize the fixed points as generated by the left hand side of \eqref{id:2measure-P}, i.e., the enumerator of ordinary partitions with respect to the $2$-measure, the length, and the size. We leave the details to the interested reader.
\end{remark}

\begin{example}
Among all the pairs $(\la,\eta)$ with $|\la|+|\eta|=6$, the fixed points are $(\epsilon, 6_{x})$, $(\epsilon, 5_x+1_x)$, $(\epsilon, 4_x+2_x)$, and $(2, 3_x+1_x)$. The remaining ones are paired up via our involution $\varphi$ as follows.
\medskip
\begin{minipage}[c]{0.5\textwidth}
\centering
\begin{tabular}{r|l}
$-$ & $+$\\
$(\epsilon, 6)$ & $(6, \epsilon)$\\
$(\epsilon, 5_x+1)$ & $(1, 5_{x})$\\
$(\epsilon, 5+1_{x})$ & $(5, 1_{x})$\\
$(\epsilon, 4_x+2)$ & $(2, 4_x)$\\
$(\epsilon, 4+2_x)$ & $(4, 2_x)$\\
$(\epsilon, 3_x+3)$ & $(3, 3_x)$\\
$(\epsilon, 3+2+1)$ & $(1, 3+2)$\\
$(\epsilon, 4+1^2)$ & $(1, 4+1)$\\
$(\epsilon, 4_x+1_x+1)$ & $(1, 4_x+1_x)$\\
$(\epsilon, 3_x+1^3)$ & $(1, 3_x+1^2)$\\
$(\epsilon, 3+1_x+1^2)$ & $(1, 3+1_x+1)$\\
$(\epsilon, 2^3)$ & $(2, 2^2)$\\
$(\epsilon, 2_x+2+1^2)$ & $(1, 2_x+2+1)$\\
$(\epsilon, 2+1^4)$ & $(1, 2+1^3)$\\
$(\epsilon, 1_x+1^5)$ & $(1, 1_x+1^4)$\\
$(1, 5)$ & $(\epsilon, 5+1)$\\
$(1, 4_x+1)$ & $(\epsilon, 4_x+1^2)$\\
$(1, 4+1_x)$ & $(\epsilon, 4+1_x+1)$\\
$(1, 3_x+2)$ & $(\epsilon, 3_x+2+1)$\\
$(1, 3+1^2)$ & $(\epsilon, 3+1^3)$\\
$(1, 3_x+1_x+1)$ & $(\epsilon, 3_x+1_x+1^2)$\\
\end{tabular}
\end{minipage}
\begin{minipage}[c]{0.5\textwidth}
\centering
\begin{tabular}{r|l}
$-$ & $+$\\
$(1, 2^2+1)$ & $(\epsilon, 2^2+1^2)$\\
$(1, 2_x+1^3)$ & $(\epsilon, 2_x+1^4)$\\
$(1, 1^5)$ & $(\epsilon, 1^6)$\\
$(2, 4)$ & $(\epsilon, 4+2)$\\
$(2, 3_x+1)$ & $(2+1, 3_x)$\\
$(2, 3+1_x)$ & $(3+2, 1_x)$\\
$(2, 2_x+2)$ & $(\epsilon, 2_x+2^2)$\\
$(2, 2+1^2)$ & $(2+1, 2+1)$\\
$(2, 1_x+1^3)$ & $(2+1, 1_x+1^2)$\\
$(3, 3)$ & $(\epsilon, 3^2)$\\
$(3, 2_x+1)$ & $(3+1, 2_x)$\\
$(3, 1^3)$ & $(3+1, 1^2)$\\
$(4, 2)$ & $(4+2, \epsilon)$\\
$(4, 1_x+1)$ & $(4+1, 1_x)$\\
$(5, 1)$ & $(5+1, \epsilon)$\\
$(2+1, 3)$ & $(2, 3+1)$\\
$(2+1, 2_x+1)$ & $(2, 2_x+1^2)$\\
$(2+1, 1^3)$ & $(2, 1^4)$\\
$(3+1, 2)$ & $(3, 2+1)$\\
$(3+1, 1_x+1)$ & $(3, 1_x+1^2)$\\
$(4+1, 1)$ & $(4, 1^2)$\\
$(3+2, 1)$ & $(3+2+1, \epsilon)$\\
\end{tabular}
\end{minipage}
\end{example}

\section{A proof of Theorem~\ref{thm:ref-d}}\label{sec:pf thm1}
The proof given here is via generating functions. Namely, we derive the generating functions of partitions from $\Ao_{n,k,m}$ and $\A2_{n,k,m}$, respectively, then connect them with the double-series expression of $D^{\sol,\ell}(x,y;q)$ given in \eqref{gf:l-sol}, and compare the coefficients to finish the proof.

Suppose $\la$ is a partition into odd parts. For our purpose, it is convenient to draw the $2$-modular diagram of $\la$ in such a way that the right border of the Durfee square of $[\la]_2$ is filled with $1$'s, instead of the usual convention to place $1$'s in the last cell of each row. So for $\la=9+7^2+5+1^2$, its $2$-modular diagram is drawn as below. We note in passing that this little ``twist'' in our way of drawing the Ferrers diagram were previously exploited by both Zeng~\cite[cf. Fig.~2]{zen05} and Alladi~\cite[cf. Fig.~1 and Fig.~2]{all20}.

\begin{figure}[ht]
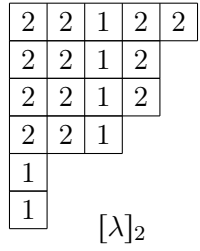

\begin{ferrers}
\addcellrows{5+4+4+3+1+1}
\highlightcellbyletter{1}{1}{2}
\highlightcellbyletter{1}{2}{2}
\highlightcellbyletter{1}{3}{1}
\highlightcellbyletter{1}{4}{2}
\highlightcellbyletter{1}{5}{2}
\highlightcellbyletter{2}{1}{2}
\highlightcellbyletter{2}{2}{2}
\highlightcellbyletter{2}{3}{1}
\highlightcellbyletter{2}{4}{2}
\highlightcellbyletter{3}{1}{2}
\highlightcellbyletter{3}{2}{2}
\highlightcellbyletter{3}{3}{1}
\highlightcellbyletter{3}{4}{2}
\highlightcellbyletter{4}{1}{2}
\highlightcellbyletter{4}{2}{2}
\highlightcellbyletter{4}{3}{1}
\highlightcellbyletter{5}{1}{1}
\highlightcellbyletter{6}{1}{1}
\addtext{1.5}{-3}{$[\la]_2$}
\end{ferrers}
\caption{The $2$-modular diagram of $\lambda=9+7^2+5+1^2$ with right border $1$'s}
\label{fig:right border}
\end{figure}

Let $\la$ be a non-empty partition into odd parts with $\Dur_2(\la)=k$ (so $k\ge 1$), we consider the following two cases. Remember that we analyze $2$-modular diagrams drawn with right border $1$'s as in Fig.~\ref{fig:right border}, so in particular, the Durfee square contributes $k(2k-1)$ to the total weight $\abs{\la}$.

\begin{enumerate}[(I)]
	\item $\la_k>2k-1$, so $[\la]_2$ is of type I. We can write $[\la]_2$ as the triple $(k;\alpha,\beta)$, where $\abs{\la}=k(2k-1)+\abs{\alpha}+\abs{\beta}$, $\alpha$ is a partition into exactly $k$ even parts, and $\beta$ is a partition into odd parts not exceeding $2k-1$. Clearly $\alpha$ is generated by $q^{2k}/(q^2;q^2)_k$, while $\beta$ is generated, according to $\dur_2(\la)$ (i.e., the Durfee side of $[\beta]_2$), by
	$$\sum_{0\le m\le k}\frac{x^mq^{m(2m-1)}}{(q;q^2)_m}\qbin{k}{m}{q^2}.$$
	Recall that $\qbin{a}{b}{q}=\dfrac{(q;q)_a}{(q;q)_b(q;q)_{a-b}}$ is the $q$-binomial coefficient and it generates all partitions whose Ferrers diagrams are confined in the $b\times(b-a)$ rectangle. Putting everything together we have
	$$1+\sum_{k\ge 1,m\ge 0}A_1(n,k,m)x^my^kq^n=\sum_{0\le m\le k}\frac{x^my^kq^{m(2m-1)+k(2k+1)}}{(q;q)_{2m}(q^2;q^2)_{k-m}}.$$
	\item $\la_k=2k-1$, so $[\la]_2$ is of type II. Again we write $[\la]_2$ as the triple $(k;\alpha,\beta)$, where $\alpha$ is a partition into at most $k-1$ even parts, generated by $1/(q^2;q^2)_{k-1}$, while $\beta$ is generated, according to $\dur_2(\la)$ (i.e., the largest $m$ such that an $m\times (m+1)$ rectangle fits in $[\beta]_2$), by
	$$\sum_{0\le m\le k-1}\frac{x^mq^{m(2m+1)}}{(q;q^2)_{m+1}}\qbin{k-1}{m}{q^2}.$$
	Altogether for this case we have
	$$\sum_{k\ge 1,m\ge 0}A_2(n,k,m)x^my^kq^n=\sum_{0\le m\le k-1}\frac{x^my^kq^{m(2m+1)+k(2k-1)}}{(q;q)_{2m+1}(q^2;q^2)_{k-m-1}}.$$
\end{enumerate}
Now we combine two cases and sum over all $k\ge 0$ to get (the term $1$ corresponds to the empty partition $\epsilon$ with $k=m=0$)
\begin{align}
&\phantom{==} 1+\sum_{k\ge 1}y^k \sum_{0\le m\le k}\frac{x^mq^{m(2m-1)+k(2k+1)}}{(q;q)_{2m}(q^2;q^2)_{k-m}}+\sum_{k\ge 1}y^k \sum_{1\le m\le k}\frac{x^{m-1}q^{(m-1)(2m-1)+k(2k-1)}}{(q;q)_{2m-1}(q^2;q^2)_{k-m}} \nonumber \\
&=\sum_{m\ge 0}\sum_{j\ge 0}\frac{x^my^{m+j}q^{4m^2+4mj+2j^2+j}}{(q;q)_{2m}(q^2;q^2)_j}+\sum_{m\ge 1}\sum_{j\ge 0}\frac{x^{m-1}y^{m+j}q^{(2m-1)^2+2(2m-1)j+2j^2+j}}{(q;q)_{2m-1}(q^2;q^2)_j} \label{gf:two cases} \\
&=\sum_{\text{even }i\ge 0}\sum_{j\ge 0}\frac{x^{\frac{i}{2}}y^{\frac{i}{2}+j}q^{i^2+2ij+2j^2+j}}{(q;q)_i(q^2;q^2)_j}+\sum_{\text{odd }i> 0}\sum_{j\ge 0}\frac{x^{\frac{i-1}{2}}y^{\frac{i+1}{2}+j}q^{i^2+2ij+2j^2+j}}{(q;q)_i(q^2;q^2)_j}.\label{gf:even-odd}
\end{align}
The right hand side of \eqref{gf:l-sol} clearly splits as
$$\sum_{i,j\ge 0,\text{ $i$ even}}\frac{x^iy^{i+2j}q^{i^2+2ij+2j^2+j}}{(q;q)_i(q^2;q^2)_j}+\sum_{i,j\ge 0,\text{ $i$ odd}}\frac{x^iy^{i+2j}q^{i^2+2ij+2j^2+j}}{(q;q)_i(q^2;q^2)_j}.$$
Comparing this with \eqref{gf:even-odd} reveals the appropriate changes of variables for $x,y$, and establishes the two identities in theorem~\ref{thm:ref-d}.

\section{Two proofs of Theorem~\ref{thm:ref-alt}}\label{sec:pf thm2}

In this section, we provide two proofs of theorem~\ref{thm:ref-alt}. The first proof boils down to a closer analysis of Sylvester's bijective proof of Euler's theorem. We basicly follow the description given by Bessenrodt in \cite{bes94} and briefly recall it here. 

Given an odd partition $\la$ with $\Dur_2(\la)=k$, we denote the {\it hooks} in the $2$-modular diagram $[\la]_2$ as $h_1^{\la},\ldots,h_k^{\la}$. The largest hook $h_1^{\la}$ traverses the northwest border of the diagram and we think of it as being removed from $\la$, then the next hook $h_2^{\la}$ traverses the northwest border of what remained. So on and so forth, until we reach the innermost hook $h_k^{\la}$, with the cells along the main diagonal of $[\la]_2$ serving as the pivot for each hook. We have colored the hooks in the $2$-modular diagram in Fig.~\ref{fig:hook} below for better illustration. Furthermore, we denote by $l_1(h_i^{\la})$ the length of $h_i^{\la}$ and by $l_2(h_i^{\la})$ the number of $2$'s covered by the hook $h_i^{\la}$. Note that $l_2(h_k^{\la})$ could be zero since the last hook might contain no $2$'s. With the abbreviation $l_{2i-1}=l_1(h_i^{\la})$ and $l_{2i}=l_2(h_i^{\la})$, the sequence $S(\la):=(l_1,l_2,l_3,\ldots,l_{2k})$ is then taken to be the image (distinct) partition (dropping the last part zero if necessary) under Sylvester's map $S$. For the odd partition $\la=9+7^2+5+1^2$ depicted in Fig.~\ref{fig:hook}, its image is seen to be $S(\la)=(10,7,5,4,3,1)$.

\begin{figure}[ht]
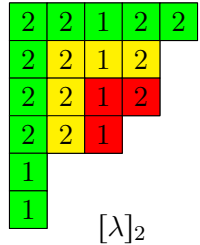

\begin{ferrers}
\addcellrows{5+4+4+3+1+1}
\addtext{1.5}{-3}{$[\la]_2$}
\highlightcellbycolor{1}{1}{green}
\highlightcellbycolor{1}{2}{green}
\highlightcellbycolor{1}{3}{green}
\highlightcellbycolor{1}{4}{green}
\highlightcellbycolor{1}{5}{green}
\highlightcellbycolor{2}{1}{green}
\highlightcellbycolor{3}{1}{green}
\highlightcellbycolor{4}{1}{green}
\highlightcellbycolor{5}{1}{green}
\highlightcellbycolor{6}{1}{green}
\highlightcellbycolor{2}{2}{yellow}
\highlightcellbycolor{2}{3}{yellow}
\highlightcellbycolor{2}{4}{yellow}
\highlightcellbycolor{3}{2}{yellow}
\highlightcellbycolor{4}{2}{yellow}
\highlightcellbycolor{3}{3}{red}
\highlightcellbycolor{4}{3}{red}
\highlightcellbycolor{3}{4}{red}

\highlightcellbyletter{1}{1}{2}
\highlightcellbyletter{1}{2}{2}
\highlightcellbyletter{1}{3}{1}
\highlightcellbyletter{1}{4}{2}
\highlightcellbyletter{1}{5}{2}
\highlightcellbyletter{2}{1}{2}
\highlightcellbyletter{2}{2}{2}
\highlightcellbyletter{2}{3}{1}
\highlightcellbyletter{2}{4}{2}
\highlightcellbyletter{3}{1}{2}
\highlightcellbyletter{3}{2}{2}
\highlightcellbyletter{3}{3}{1}
\highlightcellbyletter{3}{4}{2}
\highlightcellbyletter{4}{1}{2}
\highlightcellbyletter{4}{2}{2}
\highlightcellbyletter{4}{3}{1}
\highlightcellbyletter{5}{1}{1}
\highlightcellbyletter{6}{1}{1}
\end{ferrers}
\caption{The hooks in the $2$-modular diagram of $\lambda=9+7^2+5+1^2$}
\label{fig:hook}
\end{figure}

In order to derive an explicit formula for $S(\la)$, Bessenrodt~\cite[pp.~6]{bes94} observed the following subtle property of Sylvester's bijection.
\begin{align*}
& l_{2i-1}-l_{2i}-1=f_{2i-1} \text{ for $1\le i\le k-1$,}\\
& l_{2k-1}-l_{2k}-1=f_{2k-1} \text{ if $l_{2k}\neq 0$ and then $\la_k>2k-1$,}\\
& l_{2i}-l_{2i+1}-1=\frac{1}{2}(\la_i-\la_{i+1}) \text{ for $1\le i\le k-1$,}
\end{align*}
where $f_j$ is the number of occurrences of part $j$ in $\la$. Explained in words, the formulas above tell us the following facts.
\begin{Fact} \label{fact:seq}
For an odd partition $\la$ and its image $S(\la)=(l_1,l_2,\ldots,l_{2k})$, we have for $1\le i\le k-1$,
\begin{enumerate}[(i)]
	\item $l_{2i}$ begins a new sequence (viz., $l_{2i-1}-l_{2i}\ge 2$) if and only if there is a `small' part $2i-1$ in $\la$;
	\item $l_{2i+1}$ begins a new sequence if and only if there is also a gap $\ge 1$ between the `large' parts $\la_i$ and $\la_{i+1}$.
	\item If $l_{2k}>0$ then $l_{2k}$ begins a new sequence if and only if $2k-1$ is a part of $\la$.
\end{enumerate}
\end{Fact}

\begin{Def}[the parity index]\label{def:par}
For any sequence $s=(s_1,s_2,\ldots,s_l)$ of integers, we define its {\it parity index}, denoted as $\pai(s)$, to be the number of times the associated sequence $$s_0=0,s_1,s_2,\ldots,s_l$$ switches parities when we scan it from left to right. In particular, the parity index of $s$ is zero if it consists exlusively of even numbers.
\end{Def}

Basing on fact~\ref{fact:seq} and definition~\ref{def:par}, we make the following key definition to capture the statistic that corresponds to $\sol(S(\la))$ under Sylvester's mapping $S$. This is also the last missing piece before we give our first proof of theorem~\ref{thm:ref-alt}.
\begin{Def}[the alternating index]\label{def:alt}
Given an odd partition $\la$ represented using its $2$-modular diagram with right border $1$'s (such as in Fig.~\ref{fig:right border}), we denote it by the triple $(k;\alpha,\beta)$ and let $\eta:=\alpha^*\bigcup\beta$, where $\alpha^{*}$ is the partition we get from conjugating ~\cite[Defn.~1.8]{andtp} the $2$-modular diagram of $\alpha$. Now consider the sequence
$$\tilde{\eta}:=(\eta_1,\ldots,\eta_s,\eta_{s+1})_{\le},$$
where $(\eta_1,\ldots,\eta_s)$ are the original parts of $\eta$ in non-decreasing order, while $\eta_{s+1}$ equals $2k$ if $\la$ is of type I (viz., $\la_k>2k-1$) and it equals $2k-1$ otherwise.
We define the {\it alternating index} of $\la$, denoted as $\alt(\la)$, to be the parity index of $\tilde{\eta}$.
\end{Def}

For our running example $\la=9+7^2+5+1^2$, we easily derive from Fig.~\ref{fig:right border} its triple representation $(3;4+2+2,5+1+1)$, and thus $\eta=6+5+2+1+1$, yielding the sequence $$\tilde{\eta}=(1,1,2,5,6,6).$$
Clearly it changes parity for four times, so $\alt(\la)=\pai(\tilde{\eta})=4$, which agrees with the number of sequences of odd length in its image $S(\la)=10+7+5+4+3+1$. This is of course no coincidence. 

\begin{proof}[1st proof of theorem~\ref{thm:ref-alt}]
It suffices to demonstrate that Sylvester's mapping $S$ is a weight preserving bijection that sends an odd partition $\la$ to a distinct partition $\rho:=S(\la)$ such that
\begin{align}
& \Dur_2(\la)=\lceil\frac{\ell(\rho)}{2}\rceil, \text{ and}\label{Dur-ell}\\
& \alt(\la)=\sol(\rho).\label{alt-sol}
\end{align}
Since \eqref{Dur-ell} is already noted in \cite[Prop.~2.2(i)]{bes94}, we only need to explain \eqref{alt-sol} here. 

By our construction of $\eta$ and fact~\ref{fact:seq}, we see $\eta$ contains an odd part $2i-1$ (for $1\le i\le k-1$) if and only if $l_{2i-1}-l_{2i}\ge 2$, in which case a sequence in $\rho$ ends at $l_{2i-1}$, an odd-indexed position, while $\eta$ contains an even part $2i$ (for $1\le i\le k-1$) if and only if $l_{2i}-l_{2i+1}\ge 2$, in which case a sequence in $\rho$ ends at $l_{2i}$, an even-indexed position. Note that for the statistic $\sol(\rho)$, we only care about sequences of odd length (``odd sequence'' for short) and ignore those with even length. Therefore, we get the first odd sequence in $\rho$ as soon as we encounter the smallest odd part of $\eta$. Then the next part we count must be even, corresponding to a new odd sequence in $\rho$, so on and so forth. The parity of the part that we care indeed alternates. Appending $\eta_{s+1}$ in the end is necessary to cope with the two possible cases separately, namely the case with $\la_{k}>2k-1$ (or $l_{2k}>0$ and it ends the last sequence of $\rho$) and the case with $\la_k=2k-1$ (or $l_{2k}=0$ hence $l_{2k-1}$ actually ends the last sequence of $\rho$). In summary, each parity change in $\tilde{\eta}$ corresponds to the termination of an odd sequence in $\rho$, thus \eqref{alt-sol} holds indeed.
\end{proof}

Alternatively, we can derive the generating function 
$$B^{\alt,\Dur_2}(x,y;q):=\sum_{\la\in\cO}x^{\alt(\la)}y^{\Dur_2(\la)}q^{\abs{\la}}=\sum_{n,k,m\ge 0}B(n,k,m)x^my^kq^n$$
according to definition~\ref{def:alt}, here $\cO$ denotes the set of partitions into odd parts. The following lemma is a key ingredient in our second proof of theorem~\ref{thm:ref-alt}.

\begin{lemma}\label{lem:gf-alt}
For a positive integer $m$, let $\cP_m$ denote the set of partitions whose largest part is exactly $m$, then we have the following identity for the generating function over $\cP_m$:
\begin{align}
\label{id:gf-alt}
\sum_{\la\in\cP_m}x^{\pai(\la)}q^{\abs{\la}} &=
\begin{cases}
q^{2k} \sum\limits_{0\le j\le k}\dfrac{x^{2j}q^{\binom{2j}{2}}}{(q;q)_{2j}(q^2;q^2)_{k-j}}, & \text{if $m=2k$ is even,}\\
q^{2k-1} \sum\limits_{1\le j\le k}\dfrac{x^{2j-1}q^{\binom{2j-1}{2}}}{(q;q)_{2j-1}(q^2;q^2)_{k-j}}, & \text{if $m=2k-1$ is odd.}\\
\end{cases}
\end{align}
\end{lemma}
\begin{proof}
We only show the $m=2k$ even case, the $m$ odd case can be derived analogously and is thus omitted. Take any $\la\in\cP_{2k}$, its parity index must be an even number, say $\pai(\la)=2j$, $0\le j\le k$. Suppose $l=\ell(\la)$ and we scan the sequence 
$$(\la_{l+1}:=0,\la_l,\la_{l-1},\ldots,\la_1)$$ 
from left to right. We observe a parity change, say going from $\la_{i+1}$ to $\la_{i}$, if and only if $\la_{i}-\la_{i+1}$ is odd. In this caes we subtract $1$ from each of $\la_1,\la_2,\ldots,\la_{i}$ and we record this ``odd gap'' by adding a part $i$ to $\sigma$ (initially we set $\sigma=\epsilon$). After all the odd gaps have been handled and properly recorded in $\sigma$, we see that $\sigma$ becomes a partition into exactly $\pai(\la)=2j$ distinct parts, and we denote the partition that remained from $\la$ as $\tau$, which is seen to be a partition into even parts with the largest part being $2k-2j$. This decomposition of $\la$ into $\sigma$ and $\tau$ is invertible. Namely, $\la'=\tau'\bigcup\sigma$, where $\tau'$ is the conjugate partition of $\tau$; see Fig.~\ref{fig:decomp} below for a concrete example of this decomposition. Now $\sigma$ is generated by $q^{\binom{2j+1}{2}}/(q;q)_{2j}$ while $\tau$ is generated by $q^{2k-2j}/(q^2;q^2)_{k-j}$. Multiplying them together and summing over all possible values of $j$, we have
$$\sum_{\la\in\cP_{2k}}x^{\pai(\la)}q^{\abs{\la}}=\sum_{0\le j\le k}x^{2j}\cdot \frac{q^{\binom{2j+1}{2}}}{(q;q)_{2j}}\cdot \frac{q^{2k-2j}}{(q^2;q^2)_{k-j}}=q^{2k}\sum_{0\le j\le k}\frac{x^{2j}q^{\binom{2j}{2}}}{(q;q)_{2j}(q^2;q^2)_{k-j}}.$$
\end{proof}

\begin{figure}[ht]
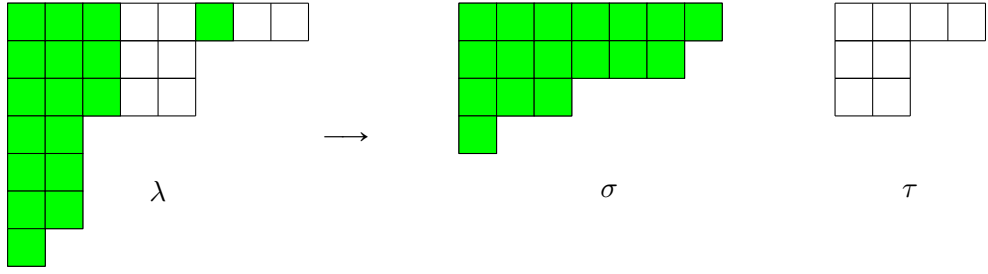

\begin{ferrers}
\addcellrows{8+5+5+2+2+2+1}
\addtext{2}{-2.5}{$\la$}
\addtext{4.5}{-1.8}{$\longrightarrow$}
\highlightcellbycolor{1}{1}{green}
\highlightcellbycolor{2}{1}{green}
\highlightcellbycolor{3}{1}{green}
\highlightcellbycolor{4}{1}{green}
\highlightcellbycolor{5}{1}{green}
\highlightcellbycolor{6}{1}{green}
\highlightcellbycolor{7}{1}{green}
\highlightcellbycolor{1}{2}{green}
\highlightcellbycolor{2}{2}{green}
\highlightcellbycolor{3}{2}{green}
\highlightcellbycolor{4}{2}{green}
\highlightcellbycolor{5}{2}{green}
\highlightcellbycolor{6}{2}{green}
\highlightcellbycolor{1}{3}{green}
\highlightcellbycolor{2}{3}{green}
\highlightcellbycolor{3}{3}{green}
\highlightcellbycolor{1}{6}{green}

\putright
\addcellrows{7+6+3+1}
\addtext{2}{-2.5}{$\sigma$}
\highlightcellbycolor{1}{1}{green}
\highlightcellbycolor{1}{2}{green}
\highlightcellbycolor{1}{3}{green}
\highlightcellbycolor{1}{4}{green}
\highlightcellbycolor{1}{5}{green}
\highlightcellbycolor{1}{6}{green}
\highlightcellbycolor{1}{7}{green}
\highlightcellbycolor{2}{1}{green}
\highlightcellbycolor{2}{2}{green}
\highlightcellbycolor{2}{3}{green}
\highlightcellbycolor{2}{4}{green}
\highlightcellbycolor{2}{5}{green}
\highlightcellbycolor{2}{6}{green}
\highlightcellbycolor{3}{1}{green}
\highlightcellbycolor{3}{2}{green}
\highlightcellbycolor{3}{3}{green}
\highlightcellbycolor{4}{1}{green}

\putright
\addcellrows{4+2+2}
\addtext{1}{-2.5}{$\tau$}
\end{ferrers}
\caption{The decomposition of $\la=8+5^2+2^3+1$ into $\sigma$ and $\tau$}
\label{fig:decomp}
\end{figure}

\begin{proof}[2nd proof of theorem~\ref{thm:ref-alt}]
Analyzing odd partitions via their $2$-modular diagrams and according to their types (either type I or type II), we derive an expression for $B^{\alt,\Dur_2}(x,y;q)$ that parallels \eqref{gf:two cases}. Given a partition $\la\in\cO$ with $\Dur_2(\la)=k$, recall the triple notation $\la=(k;\alpha,\beta)$ where $\alpha$ is a partition into at most $k$ even parts, and $\beta$ is a partition into odd parts not exceeding $2k-1$. Recall that we have defined $\eta=\alpha^*\bigcup\beta$ in definition~\ref{def:alt}. The rest of the proof splits into two cases.
\begin{enumerate}[(1)]
	\item If $[\la]_2$ is of type I, i.e., $\la_k>2k-1$, then $\alpha$ is a partition with exactly $k$ even parts. Consequently the union $\eta$ is a partition whose largest part is exactly $2k$. This implies in particular that $\pai(\tilde{\eta})=\pai(\eta)$ must be even. We apply \eqref{id:gf-alt} for $m=2k$ to deduce that $\eta$ is generated with respect to its parity index by
	$$q^{2k}\sum_{0\le j\le k}\frac{x^{2j}q^{2j^2-j}}{(q;q)_{2j}(q^2;q^2)_{k-j}}.$$
	Together with the contribution from the $2$-modular Durfee square we have the generating function for all odd partitions of type I.
	\begin{align}\label{gf:type I}
	y^{k}q^{k(2k+1)}\sum_{0\le j\le k}\frac{x^{2j}q^{2j^2-j}}{(q;q)_{2j}(q^2;q^2)_{k-j}}.
	\end{align}
	\item If $[\la]_2$ is of type II, i.e., $\la_k=2k-1$, then $\alpha$ is a partition with at most $k-1$ even parts. Consequently $\tilde{\eta}$ is a partition whose largest part is $2k-1$ and $\pai(\tilde{\eta})$ must be odd. We apply \eqref{id:gf-alt} for $m=2k-1$ to deduce that $\tilde{\eta}$ is generated by
	$$q^{2k-1}\sum_{1\le j\le k}\frac{x^{2j-1}q^{\binom{2j-1}{2}}}{(q;q)_{2j-1}(q^2;q^2)_{k-j}}.$$
	Hence odd partitions of type II are generated by (we need to multiply $q^{-(2k-1)}$ going from $\tilde{\eta}$ back to $\eta$)
	\begin{align}\label{gf:type II}
	y^kq^{k(2k-1)}\sum_{1\le j\le k}\frac{x^{2j-1}q^{\binom{2j-1}{2}}}{(q;q)_{2j-1}(q^2;q^2)_{k-j}}.
	\end{align}
\end{enumerate}
Now putting together \eqref{gf:type I} and \eqref{gf:type II} and summing over all $k\ge 1$, we get (again the term $1$ corresponds to the empty partition $\epsilon$):
\begin{align*}
B^{\alt,\Dur_2}(x,y;q) &= 1+\sum_{k\ge 1}y^k\left(\sum_{0\le j\le k}\frac{x^{2j}q^{2k^2+2j^2+k-j}}{(q;q)_{2j}(q^2;q^2)_{k-j}}+\sum_{1\le j\le k}\frac{x^{2j-1}q^{2k^2+2j^2-k-j}}{(q;q)_{2j-1}(q^2;q^2)_{k-j}}\right)\\
\text{\tiny $(j\to m,~k\to m+j)$} \: &= \sum_{m,j\ge 0}\frac{x^{2m}y^{m+j}q^{4m^2+4mj+2j^2+j}}{(q;q)_{2m}(q^2;q^2)_j}+\sum_{m\ge 1,\: j\ge 0}\frac{x^{2m-1}y^{m+j}q^{(2m-1)^2+2(2m-1)j+2j^2+j}}{(q;q)_{2m-1}(q^2;q^2)_j}.
\end{align*}
Comparing this last expression with \eqref{gf:two cases} results in $A(n,k,\lfloor m/2 \rfloor)=B(n,k,m)$, which proves theorem~\ref{thm:ref-alt} in view of theorem~\ref{thm:ref-d}.
\end{proof}

\section{Concluding remarks}\label{sec:conclusion}

After viewing the two proofs of theorem~\ref{thm:ref-alt} given in section~\ref{sec:pf thm2}, one naturally wonders if there exists a direct bijection transforming a partition $\la$ of $n$ into odd parts to a strict partition $\mu$ of $n$, such that $\Dur_2(\la)=\lceil\ell(\mu)/2\rceil$ and $\dur_2(\la)=\lfloor\sol(\mu)/2\rfloor$. Once found, this map will effectively prove theorem~\ref{thm:ref-d} bijectively. Besides Sylvester's bijection that proves theorem~\ref{thm:ref-alt}, there are at least two other bijections between odd partitions and strict partitions that we are aware of (see for instance Pak's survey~\cite[Sect.~3]{pak06}). They are Glaisher's map $\varphi_{g}$ and the iterated Dyson's map $\xi$. Unfortunately, neither of them preserves the statistics as required above. We give one example to illustrate this. For the partition $\la=11+3+1\in\A2_{15,2,0}$, we have
\begin{align*}
& 11+3+1\stackrel{\varphi_{g}}{\mapsto} 11+3+1\not\in\cD_{15,3,1},\\
& 11+3+1\stackrel{\xi}{\mapsto} 12+3\not\in\cD_{15,3,1}.
\end{align*}

In both \cite{FL24} and the current work, the statistic $\sol$ has been restricted to the set of strict partitions. Can a similar notion of ``$\sol$'' be introduced for other classes of restricted partitions, so as to give partition theoretical interpretations to other multi-sum $q$-series identities? Basing on our initial observations, several identities from recent works of Wang~\cite{wan24} and Li-Wang~\cite{LW24} are potential candidates for such a treatment. We aim to address them in our upcoming work~\cite{FL25}.

\section*{Acknowledgement}
Both authors were partially supported by the National Natural Science Foundation of China grant 12171059 and the Mathematical Research Center of Chongqing University.

\end{document}